\documentclass{article}
\usepackage{amsmath,amsthm,amssymb}
\usepackage{graphicx}
\usepackage{fancybox}

\theoremstyle{definition}
\newtheorem{Thm}{Theorem}[section]
\newtheorem{Def}[Thm]{Definition}
\newtheorem{Lem}[Thm]{Lemma}
\newtheorem{Cor}[Thm]{Corollary}

\newtheorem{Prop}[Thm]{Proposition}
\newtheorem{Rem}[Thm]{Remark}

\def\Z{\mathbb{Z}}
\def\bbh{\mathbb{H}}
\def\bbk{\mathbb{K}}
\def\bbl{\mathbb{L}}

\def\calk{\mathcal{K}}
\def\call{\mathcal{L}}
\def\calm{\mathcal{M}}
\def\Aut{\mathrm{Aut}}
\def\Out{\mathrm{Out}}

\title{Hyperbolically embedded virtually free subgroups of relatively hyperbolic groups}
\date{}
\author{
Yoshifumi Matsuda\footnote
{Graduate School of Mathematical Sciences, 
University of Tokyo, 
3-8-1 Komaba, 
Meguro-ku, 
Tokyo, 
153-8914 Japan,
ymatsuda@ms.u-tokyo.ac.jp} \footnote{
Supported by the Global COE Program at Graduate School of Mathematical Sciences, the University of Tokyo, and Grant-in-Aid for Scientific Researches for Young Scientists (B) (No. 22740034), Japan Society of Promotion of Science.}, 
Shin-ichi Oguni\footnote
{Department of Mathematics, Faculty of Science,
Ehime University,
2-5 Bunkyo-cho, 
Matsuyama, 
Ehime, 
790-8577 Japan,
oguni@math.sci.ehime-u.ac.jp}, 
Saeko Yamagata\footnote
{Faculty of Education and Human Sciences, Yokohama National University,
240-8501 Yokohama, Japan, 
yamagata@ynu.ac.jp}
}

\begin{document}

\maketitle
\begin{abstract} 
We show that if a group is not virtually cyclic and is hyperbolic relative to a family of proper subgroups, then it has a hyperbolically embedded subgroup which contains a finitely generated non-abelian free group as a finite index subgroup. \\


\noindent
2010MSC:
20F67;
20F65
\end{abstract}

\section{Introduction}
The notion of relatively hyperbolic groups was introduced in \cite{Gro87} and has been studied by many authors (see for example \cite{Bow12}, \cite{D-S05}, \cite{Far98} and \cite{Osi06a}). 
In this paper we consider relatively hyperbolic groups in accordance with a definition due to D. Osin \cite[Definition 2.35]{Osi06a}. 
Note that for certain cases (e.g. for finitely generated groups), this definition has several equivalent formulations (see for example \cite[Sections 3 and 5]{Hru10}). 

In \cite{Osi06b}, D. Osin introduced the notion of hyperbolically embedded subgroups of a relatively hyperbolic group. 

\begin{Def}\label{def-hypemb} (\cite[Definition 1.4]{Osi06b})
Let $G$ be a group which is hyperbolic relative to a family $\bbk$ of subgroups. A subgroup $H$ of $G$ is said to be hyperbolically embedded into $G$ relative to $\bbk$ if $G$ is hyperbolic relative to $\bbk\cup\{H\}$.
\end{Def}

If $G$ is infinite and hyperbolic relative to a family $\bbk$ of proper subgroups, then there exists a virtually infinite cyclic subgroup of $G$ which is hyperbolically embedded into $G$ relative to $\bbk$ (see \cite[Corollaries 1.7 and 4.5]{Osi06b}). 
Also if $G$ is torsion-free, hyperbolic and not cyclic, then it contains a free subgroup of rank two which is quasiconvex and malnormal in $G$, that is, hyperbolically embedded into $G$ relative to the empty family $\emptyset$ (see \cite[Theorem C]{Kap99} and \cite[Theorem 7.11]{Bow12}). 
In this paper we show the following (see also Theorem \ref{hypemb'}). 

\begin{Thm}\label{hypemb}
Suppose that a group $G$ is not virtually cyclic and is hyperbolic relative to a family $\bbk$ of proper subgroups. 
Then there exists a finitely generated and virtually non-abelian free subgroup of $G$ which is hyperbolically embedded into $G$ relative to $\bbk$. 
Moreover if $G$ is torsion-free, then it contains a free subgroup of rank two which is hyperbolically embedded into $G$ relative to $\bbk$. 
\end{Thm}

We refer to \cite[Theorems 1.4 and 1.5]{M-O-Y5} for applications of this theorem to the study of convergence actions of groups. 
We also refer to \cite[Theorem 6.3]{M-O-Y3} for another application. 

\begin{Rem}\label{rem-DGO}
After the first version of this paper appeared, the notion of a hyperbolically embedded subgroup was further generalized in \cite{D-G-O11}. 
It turns out that we can prove a stronger version of Theorem \ref{hypemb} by using the argument in the proof of \cite[Theorem 6.14 (c)]{D-G-O11} (see \cite[Appendix B]{M-O-Y5} for details). 
In what follows, hyperbolically embedded subgroups which we consider are those in the sense of \cite{Osi06b}. 
\end{Rem}

In Section \ref{sect-char}, we recall the fact that hyperbolically embedded subgroups of a relatively hyperbolic group are characterized as strongly relatively undistorted and almost malnormal subgroups. 
Strongly relatively undistorted free subgroups of rank two of a relatively hyperbolic group are found in Section \ref{sect-sruf}. 
In Section \ref{sect-amvf}, we construct almost malnormal subgroups of a virtually free group with additional properties. 
Theorem \ref{hypemb} is proved in Section \ref{sect-proof}. 

\section{Characterization of hyperbolically embedded subgroups}
\label{sect-char}
The strategy of our proof of Theorem \ref{hypemb} is based on Osin's characterization of hyperbolically embedded subgroups of relatively hyperbolic groups stated below. 

To state the characterization, we begin by introducing several definitions. 
Let $G$ be a group. 
For a family $\bbk$ of subgroups of $G$, we put $\calk=\bigcup_{K\in\bbk} K\setminus\{1\}$. 
A subset $X$ of $G$ is called a relative generating set of $G$ with respect to $\bbk$ if $G$ is generated by $X\cup\calk$. 
The group $G$ is said to be finitely generated relative to $\bbk$ if there exists a finite relative generating set of $G$ with respect to $\bbk$. 
When $Z$ is a (possibly infinite) generating set of $G$, we denote by $\Gamma(G,Z)$ the Cayley graph of $G$ with respect to $Z$ and by $d_Z$ the word metric with respect to $Z$. 

\begin{Def}\label{def-sru}
Let $G$ be a group which is finitely generated relative to a family $\bbk$ of subgroups. 
A subgroup $H$ of $G$ is said to be strongly undistorted relative to $\bbk$ in $G$ if $H$ is generated by some finite subset $Y$ and for some finite relative generating set $X$ of $G$ with respect to $\bbk$, the natural map $(H,d_Y)\to (G,d_{X\cup\calk})$ is a quasi-isometric embedding. 

\end{Def}

\begin{Def}\label{def-am}
Let $G$ be a group and $H$ a subgroup of $G$. 
The subgroup $H$ is said to be malnormal (resp. almost malnormal) in $G$ if for every element $g$ of $G\setminus H$, the intersection $H \cap gHg^{-1}$ is trivial (resp. finite). 
\end{Def}

\begin{Thm}\label{hypemb-sruam} (\cite[Theorem 1.5]{Osi06b})
Let $G$ be a group which is hyperbolic relative to a family $\bbk$ of subgroups. Then a subgroup $H$ of $G$ is hyperbolically embedded into $G$ relative to $\bbk$ if and only if $H$ is strongly undistorted relative to $\bbk$ and almost malnormal in $G$. 
\end{Thm}

For finitely generated relatively hyperbolic groups, we have the following characterization of strongly relatively undistorted subgroups. 

\begin{Def}\label{def-sq} (\cite[Definitions 4.9 and 4.11]{Osi06a})
Let $G$ be a group with a finite family $\bbk$ of subgroups. 
Suppose that $G$ is generated by a finite set $X$. 
A subgroup $H$ of $G$ is said to be quasiconvex relative to $\bbk$ in $G$ if there exists a constant $\sigma\ge 0$ satisfying the following: 
if $h_1$ and $h_2$ are elements of $H$ and $p$ is a geodesic from $h_1$ to $h_2$ in $\Gamma(G,X\cup\calk)$, then for every vertex $v$ on $p$, there exists an element $h$ of $H$ such that $d_X(v,h)\le\sigma$. 
A subgroup $H$ of $G$ is said to be strongly quasiconvex relative to $\bbk$ in $G$ if $H$ is quasiconvex relative to $\bbk$ in $G$ and for every element $K$ of $\bbk$ and every element $g$ of $G$, the intersection $H \cap gKg^{-1}$ is finite. 
\end{Def}

\begin{Thm}\label{sru-sq} (\cite[Theorem 4.13]{Osi06a}) 
Let $G$ be a group which is hyperbolic relative to a finite family $\bbk$ of subgroups. Suppose that $G$ is finitely generated. 
Then a subgroup $H$ of $G$ is strongly undistorted relative to $\bbk$ if and only if $H$ is strongly quasiconvex relative to $\bbk$. 
\end{Thm}

\section{Strongly relatively undistorted free subgroups}
\label{sect-sruf}
When a group $G$ is hyperbolic relative to a family $\bbk$ of subgroups, a subgroup of $G$ is said to be parabolic with respect to $\bbk$ if it is conjugate to a subgroup of some element of $\bbk$.
The main purpose of this section is to show the following. 

\begin{Prop}\label{sq'}
Let $G$ be a group which is hyperbolic relative to a family $\bbk$ of proper subgroups and $\Gamma$ a subgroup of $G$ which is neither virtually cyclic nor parabolic with respect to $\bbk$. 
If $\Gamma$ contains an element of infinite order, then it contains a free subgroup $F$ of rank two which is strongly undistorted relative to $\bbk$ in $G$. 
\end{Prop}

Proposition \ref{sq'} yields the following corollary. 

\begin{Cor}\label{sq}
Let $G$ be a group which is not virtually cyclic and is hyperbolic relative to a family $\bbk$ of proper subgroups. 
Then $G$ contains a free subgroup $F$ of rank two which is strongly undistorted relative to $\bbk$ in $G$. 
\end{Cor}

\begin{proof}[Proof of Corollary \ref{sq} using Proposition \ref{sq'}]
It follows from \cite[Corollary 4.5]{Osi06b} that the group $G$ contains an element of infinite order. 
Hence the assertion follows from Proposition \ref{sq'}. 
\end{proof}

For the proof of Proposition \ref{sq'}, we prepare several lemmas. 

When a group $G$ is hyperbolic relative to a family $\bbk$ of proper subgroups, an element $g$ of $G$ is said to be parabolic with respect to $\bbk$ if it is conjugate to an element of a subgroup of $G$ which belongs to $\bbk$. 
Otherwise $g$ is said to be hyperbolic with respect to $\bbk$. 

\begin{Lem}\label{hyp-element}
Let $G$ be a group which is hyperbolic relative to a family $\bbk$ of proper subgroups and $\Gamma$ a subgroup of $G$ which is neither virtually cyclic nor parabolic with respect to $\bbk$. 
Suppose that either $G$ is countable and $\bbk$ is finite or $\Gamma$ contains an element of infinite order. 
Then there exists an element $h$ of $\Gamma$ which is of infinite order and hyperbolic with respect to $\bbk$. 
\end{Lem}

\begin{proof}
First suppose that $G$ is countable and $\bbk$ is finite. 
Then we can consider a geometrically finite convergence action of $G$ on a compact metrizable space such that the set of all maximal parabolic subgroups of the action is equal to the collection of all conjugates of elements of $\bbk$ which are infinite (see for example \cite[Definition 3.1]{Hru10}). 
Since $\Gamma$ is neither virtually cyclic nor parabolic with respect to $\bbk$, the restriction of this action to $\Gamma$ is a non-elementary convergence action. 
Hence $\Gamma$ contains an element $h$ which is loxodromic with respect to this action (see \cite[Theorem 2T]{Tuk94}). 

Next suppose that $\Gamma$ contains an element $h$ of infinite order. 
We only have to consider the case where $h$ belongs to the conjugate $gKg^{-1}$ for some element $K$ of $\bbk$ and some element $g$ of $G$. 
Since $\Gamma$ is not parabolic with respect to $\bbk$, we can take an element $\gamma$ of $\Gamma\setminus gKg^{-1}$. 
By \cite[Lemma 4.4]{Osi06b}, there exists an integer $n$ such that the element $\gamma h^n$ of $\Gamma$ is of infinite order and hyperbolic with respect to $\bbk$. 
\end{proof}

\begin{Lem}\label{hyp-hypemb} (\cite[Theorem 4.3 and Corollary 1.7]{Osi06b})
Let $G$ be a group which is hyperbolic relative to a family $\bbk$ of subgroups and $h$ be an element of $G$ which is of infinite order and hyperbolic with respect to $\bbk$. 
Then there exists a unique subgroup $E(h)$ of $G$ such that $E(h)$ is virtually cyclic, contains $h$ and maximal among such subgroups of $G$. 
Moreover $E(h)$ is hyperbolically embedded into $G$ relative to $\bbk$. 
\end{Lem}

The following lemma is shown by a similar argument in the proof of \cite[Corollary 1.12]{MP12}. 

\begin{Lem}\label{mp}
Let $G$ be a group generated by a finite set $X$ and hyperbolic relative to a finite family $\bbk$ of subgroups. 
Let a subgroup $H$ of $G$ be hyperbolically embedded into $G$ relative to $\bbk$. 
We denote the union $\bbk\cup\{H\}$ by $\bbh$. 
Suppose that a subgroup $Q$ of $G$ is strongly quasiconvex relative to $\bbh$ in $G$. 
Then there exists a constant $C(Q,H)\ge 0$ with the following property:
for every subgroup $R$ of $H$ such that 
\begin{enumerate}
\setlength{\itemsep}{0mm}
\item[(a)] $Q\cap H\subset R$;
\item[(b)] $d_X(1, r) \ge C(Q,H)$ for every element $r$ of $R\setminus Q$; 
\item[(c)] the subgroup $R$ is quasiconvex relative to $\bbk$ in $G$,
\end{enumerate}
the natural homomorphism $Q*_{Q\cap R}R\to G$ is injective and its image $\langle Q\cup R \rangle$ is strongly quasiconvex relative to $\bbk$ in $G$. 
\end{Lem}

\begin{proof}
By \cite[Theorem 5.12]{MP12}, there exists a constant $C(Q,H)\ge 0$ with the following property:
for every subgroup $R$ of $H$ satisfying above conditions (a) and (b), the natural homomorphism $Q*_{Q\cap R}R\to G$ is injective, its image $\langle Q\cup R \rangle$ is quasiconvex relative to $\bbh$ in $G$ and for every element $g$ of $G$ and every element $H'$ of $\bbh$, the intersection $\langle Q\cup R \rangle\cap gH'g^{-1}$ is either finite or conjugate to $R$ in $\langle Q\cup R \rangle$. 

Now we suppose that $R$ satisfies above condition (c) and show that $\langle Q\cup R \rangle$ is strongly quasiconvex relative to $\bbk$ in $G$. 

First we show that $\langle Q\cup R \rangle$ is quasiconvex relative to $\bbk$ in $G$. 
By \cite[Theorem 1.1 (2)]{MP12}, it suffices to show that for every element $g$ of $G$, the intersection $\langle Q\cup R \rangle\cap gHg^{-1}$ is quasiconvex relative to $\bbk$ in $G$. 
Every finite subgroup of $G$ is automatically quasiconvex relative to $\bbk$ in $G$. 
Since $R$ is quasiconvex relative to $\bbk$ in $G$, it is well-known that every conjugate of $R$ is also quasiconvex relative to $\bbk$ in $G$. 
Thus the subgroup $\langle Q\cup R \rangle$ is quasiconvex relative to $\bbk$ in $G$. 

Next we show that for every element $g$ of $G$ and every element $K$ of $\bbk$, the intersection $\langle Q\cup R \rangle\cap gKg^{-1}$ is finite. 
We have only to consider the case where there exists an element $s$ of $\langle Q\cup R \rangle$ such that $\langle Q\cup R \rangle\cap gKg^{-1}$ is equal to $sRs^{-1}$. 
Then $\langle Q\cup R \rangle\cap gKg^{-1}$ is contained in $s(s^{-1}gK(s^{-1}g)^{-1}\cap H)s^{-1}$. 
Since $H$ is hyperbolically embedded into $G$ relative to $\bbk$, the intersection $s^{-1}gK(s^{-1}g)^{-1}\cap H$ is finite. 
Hence $\langle Q\cup R \rangle\cap gKg^{-1}$ is also finite. 

Thus $\langle Q\cup R \rangle$ is strongly quasiconvex relative to $\bbk$ in $G$. 
\end{proof}

By using the above lemmas, we prove the following, which implies Proposition \ref{sq'} for finitely generated groups. 

\begin{Lem}\label{sq'fg}
Let $G$ be a group which is hyperbolic relative to a finite family $\bbk$ of proper subgroups and $\Gamma$ a subgroup of $G$ which is neither virtually cyclic nor parabolic with respect to $\bbk$. 
Suppose that $G$ is finitely generated. 
Then $\Gamma$ contains a free subgroup $F$ of rank two which is strongly quasiconvex relative to $\bbk$ in $G$. 
\end{Lem}

\begin{proof}
By Lemma \ref{hyp-element}, there exists an element $h$ of $\Gamma$ which is of infinite order and hyperbolic with respect to $\bbk$. 
We denote by $H$ a subgroup $E(h)$ of $G$ given by Lemma \ref{hyp-hypemb}. 
We put $\bbh=\bbk\cup\{H\}$. 
Then $\bbh$ consists of proper subgroups of $G$ and $G$ is hyperbolic relative to $\bbh$. 
Since $H$ is virtually infinite cyclic, the subgroup $\Gamma$ is not parabolic with respect to $\bbh$. 

Hence it follows from Lemma \ref{hyp-element} that there exists an element $q$ of $\Gamma$ which is of infinite order and hyperbolic with respect to $\bbh$. 
We denote by $Q$ the infinite cyclic subgroup of $\Gamma$ generated by $q$. 
By Lemma \ref{hyp-hypemb} and Theorem \ref{sru-sq}, the subgroup $Q$ is strongly quasiconvex relative to $\bbh$ in $G$. 
Since $Q$ is torsion-free, this implies that the intersection $Q\cap H$ is trivial. 

Let $X$ be a finite generating set of $G$ and $C(Q,H)\ge 0$ a constant given by Lemma \ref{mp}. 
Since $h$ is of infinite order, there exists a positive integer $k$ such that $d_X(1,h^{kn}) \ge C(Q,H)$ for every integer $n\in\Z\setminus\{0\}$. 
We denote by $R$ the infinite cyclic subgroup of $\Gamma$ generated by $h^k$. 
Since $R$ is a finite index subgroup of $H$, it is strongly quasiconvex relative to $\bbk$ in $G$ by Theorem \ref{sru-sq}. 
Hence $R$ satisfies conditions (a), (b) and (c) in Lemma \ref{mp}. 
By Lemma \ref{mp}, the subgroup $\langle Q\cup R \rangle$ of $\Gamma$ is a free group of rank two which is strongly quasiconvex relative to $\bbk$ in $G$. 
\end{proof}

\begin{figure}
\begin{center}
\includegraphics[width=4.8cm,height=2.8cm]{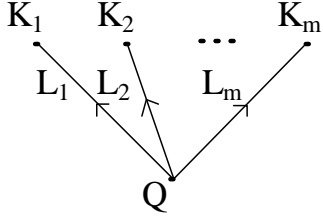}
\end{center}
\caption{The graph of groups $\mathcal{G}$}
\label{graphofgroups}
\end{figure}

For the proof of the general case, 
we need the following lemma which is obtained from a specialization of \cite[Theorem 2.44]{Osi06a} together with \cite[Proposition 2.49]{Osi06a}. 

\begin{Lem}\label{fgru}
Let $G$ be a group which is hyperbolic relative to a family $\bbk$ of proper subgroups and $X$ a finite relative generating set of $G$ with respect to $\bbk$. 
Then there exists a finite subfamily $\bbk_0=\{K_1, \ldots, K_m\}$ of $\bbk$ such that $G$ splits as the free product 
\begin{eqnarray*}
G=G_0\ast(\ast_{K\in\bbk\setminus\bbk_0}K),
\end{eqnarray*}
where $G_0$ is the subgroup of $G$ which is generated by $K_1, \ldots, K_m$ and $X$. 
Moreover there exist a finitely generated subgroup $Q$ of $G_0$ and a family $\bbl=\{L_1, \ldots, L_m\}$ of subgroups of $Q$ satisfying the following: 
\begin{enumerate}
\setlength{\itemsep}{0mm}
\item[(i)] the finite relative generating set $X$ is contained in $Q$ and for every $i \in\{1, \ldots, m\}$, the subgroup $L_i$ is contained in $K_i$; 
\item[(ii)] the group $G_0$ is isomorphic to the fundamental group of the graph of groups $\mathcal{G}$ drawn in Figure \ref{graphofgroups};
\item[(iii)] the subgroup $Q$ is hyperbolic relative to $\bbl$, the set $X$ is a relative generating set of $Q$ with respect to $\bbl$ and the natural map $(Q,d_{X\cup\call})\to (G,d_{X\cup\calk})$ is an isometric embedding, where we put $\call=\bigcup_{L\in\bbl} L\setminus\{1\}$. 
\end{enumerate}
\end{Lem}

\begin{Lem}\label{fgru'}
In the setting of Lemma \ref{fgru}, we have the following:
\begin{enumerate}
\setlength{\itemsep}{0mm}
\item[(1)] if no elements of $\bbk$ contain $X$, then $\bbl$ consists of proper subgroups of $Q$. 
\item[(2)] if a subgroup of $Q$ is strongly quasiconvex relative to $\bbl$ in $Q$, then it is strongly undistorted relative to $\bbk$ in $G$;
\item[(3)] if a subgroup of $Q$ is hyperbolically embedded into $Q$ relative to $\bbl$, then it is hyperbolically embedded into $G$ relative to $\bbk$; 
\end{enumerate}
\end{Lem}

\begin{proof}
(1) This follows from condition (i) in Lemma \ref{fgru}. 

(2) This follows from Theorem \ref{sru-sq} and condition (iii) in Lemma \ref{fgru}. 

(3) Suppose that a subgroup $V$ of $Q$ is hyperbolically embedded into $Q$ relative to $\bbl$. 
Then $V$ is strongly quasiconvex relative to $\bbl$ in $Q$ by Theorems \ref{hypemb-sruam} and \ref{sru-sq}. 
By assertion (2), the subgroup $V$ is strongly undistorted relative to $\bbk$ in $G$. 

We claim that $V$ is almost malnormal in $G$. 
Indeed, since $V$ is hyperbolically embedded into $Q$ relative to $\bbl$, it is almost malnormal in $Q$. 
Hence it suffices to show that if $g$ is an element of $G\setminus Q$, then the intersection $V \cap gVg^{-1}$ is finite. 
First suppose that $g$ belongs to $G \setminus G_0$. 
Since $G_0$ is a free factor of $G$, the intersection $G_0 \cap gG_0g^{-1}$ is trivial. 
Hence the intersection $V \cap gVg^{-1}$ is also trivial. 
Next suppose that $g$ belongs to $G_0\setminus Q$. 
We denote by $T$ the Bass-Serre covering tree of the graph of groups $\mathcal{G}$. 
Then the group $Q$ is the stabilizer group of a vertex $v$ of $T$ and we have $gv \neq v$. 
Since the intersection $Q \cap gQg^{-1}$ fixes both $v$ and $gv$, it fixes an edge of $T$. 
Hence $Q \cap gQg^{-1}$ is parabolic with respect to $\bbl$. 
Since every element of $\bbl$ is contained in a element of $\bbk$, the intersection $Q \cap gQg^{-1}$ is parabolic with respect to $\bbk$. 
Since the subgroup $V$ is strongly quasiconvex relative to $\bbk$ in $G$, the intersection $V \cap (Q \cap gQg^{-1})$ is finite. 
Hence the intersection $V \cap gVg^{-1}$ is also finite. 
\end{proof}

\begin{proof}[Proof of Proposition \ref{sq'}]
Since $\Gamma$ contains an element of infinite order, it follows from Lemma \ref{hyp-element} that there exists an element $h$ of $\Gamma$ which is of infinite order and hyperbolic with respect to $\bbk$. 
Let $E(h)$ be a subgroup of $G$ given by Lemma \ref{hyp-hypemb}. 
Since $\Gamma$ is not virtually cyclic, we can take an element $\gamma$ of $\Gamma\setminus E(h)$. 
By Lemma \ref{hyp-hypemb}, every subgroup of $G$ that contains $\{h,\gamma\}$ is not virtually cyclic. 
We take a finite relative generating set $X$ of $G$ with respect to $\bbk$ which contains $\{h,\gamma\}$.  
Note that since $h$ is hyperbolic with respect to $\bbk$, no elements of $\bbk$ contain $X$. 

Let $Q$ and $\bbl$ be given by Lemma \ref{fgru}. 
By Lemma \ref{fgru'} (1), the family $\bbl$ consists of proper subgroups of $Q$. Since $\Gamma\cap Q$ contains $\{h,\gamma\}$ and each element of $\bbl$ is contained in some element of $\bbk$, the subgroup $\Gamma\cap Q$ is neither virtually cyclic nor parabolic with respect to $\bbl$. 

Since $Q$ is finitely generated, it follows from Lemma \ref{sq'fg} that $\Gamma\cap Q$ contains a free subgroup $F$ of rank two which is strongly quasiconvex relative to $\bbl$ in $Q$. 
By Lemma \ref{fgru'} (2), the subgroup $F$ is strongly undistorted relative to $\bbk$ in $G$. 
\end{proof}

\begin{Rem}\label{alt}
We give an alternative proof of Corollary \ref{sq}. 

Let $F'$ be a free group of rank two with a free basis $Y'$ and $X$ a finite relative generating set of $G$ with respect to $\bbk$. 
By \cite[Theorem 1.1]{A-M-O07}, there exists a quotient group $G'$ of $G$ and an embedding $\iota \colon F' \to G'$ such that $G'$ is hyperbolic relative to $\{\psi(K)\}_{K\in\bbk} \cup \{\iota(F')\}$, where $\psi \colon G \to G'$ denotes the natural projection. 
Since $\iota(F')$ is a hyperbolic group, it follows from \cite[Theorem 2.40]{Osi06a} that $G'$ is hyperbolic relative to $\{\psi(K)\}_{K\in\bbk}$. 
Hence $\iota(F')$ is hyperbolically embedded into $G'$ relative to $\{\psi(K)\}_{K\in\bbk}$. 
By Theorem \ref{hypemb-sruam}, the natural map $(F',d_{Y'}) \to (G',d_{\psi(X)\cup\calk'})$ is a quasi-isometric embedding, where $\calk'$ denotes $\bigcup_{K\in\bbk} \psi(K) \setminus\{1\}$. 
We take a subset $Y$ of $G$ such that $Y$ consists of two elements and $\psi(Y)$ is equal to $\iota(Y')$. 
We denote by $F$ the subgroup of $G$ which is generated by $Y$. 
Then $F$ is a free group of rank two. 
We can confirm that the natural map $(F,d_Y) \to (G,d_{X\cup\calk})$ is a quasi-isometric embedding. 
This finishes the proof of Corollary \ref{sq}. 
\end{Rem}

\section{Almost malnormal subgroups of virtually free groups}
\label{sect-amvf}
In this section, we show the following (compare with \cite[Theorem 5.16]{Kap99} for the case of non-abelian free groups of finite rank), 
which is necessary for the proof of Theorem \ref{hypemb}. 

\begin{Thm}\label{kap1}
Let $M$ be a finitely generated and virtually non-abelian free group and let $\{M_l ~|~ l \in \{1, \ldots, n\}\}$ be a finite family of finitely generated subgroups of $M$ of infinite index. 
Then there exists a proper subgroup $V$ of $M$ satisfying the following: 
\begin{enumerate}
\setlength{\itemsep}{0mm}
\item[(i)] the subgroup $V$ is finitely generated, virtually non-abelian free, and almost malnormal in $M$; 
\item[(ii)] for every $l \in \{1, \ldots, n\}$ and every element $m$ of $M$, the intersection $mVm^{-1} \cap M_l$ is finite. 
\end{enumerate}
\end{Thm}

For the proof of Theorem \ref{kap1}, we prepare two lemmas. 

\begin{Lem}\label{aminfind}
If a proper subgroup $H$ of an infinite group $G$ is almost malnormal in $G$, then $H$ is an infinite index subgroup of $G$. 
\end{Lem}

\begin{proof}
Assume that $H$ is a finite index subgroup of $G$. 
Then for every element $g$ of $G$, the intersection $H \cap gHg^{-1}$ is a finite index subgroup of $G$ and hence it is infinite. 
This contradicts the assumption that $H$ is almost malnormal in $G$. 
\end{proof}

\begin{Lem}\label{kap2}
Let $F$ be a non-abelian free group of finite rank and let $\{H_l ~|~ l \in \{1, \ldots, n\}\}$ be a finite family of finitely generated subgroups of $F$ of infinite index. 
Let $U$ be a finite subgroup of $\Out(F)$. 
We denote by $\pi \colon \Aut(F) \to \Out(F)$ the quotient map and put $A=\pi^{-1}(U)$. Then there exists a proper subgroup $H$ of $F$ satisfying the following: 
\begin{enumerate}
\setlength{\itemsep}{0mm}
\item[(i)] the subgroup $H$ is a free subgroup of rank two and is malnormal in $F$;
\item[(ii)] for every $l \in \{1, \ldots, n\}$ and every element $a$ of $A$, the intersection $H_l \cap a(H)$ is trivial; 
\item[(iii)] for every element $a$ of $A$, either $a(H)$ is equal to $H$ or the intersection $H \cap a(H)$ is trivial.
\end{enumerate}
\end{Lem}
\begin{proof}
We put $U=\{u_i ~|~ i \in \{1, \ldots, m\}\}$ and choose an element $a_i$ of $\pi^{-1}(u_i)$ for each $i \in \{1, \ldots, m\}$. 
We denote by $\calm$ the collection of all proper subgroups $H'$ of $F$ satisfying the following:
\begin{itemize}
\setlength{\itemsep}{0mm}
\item the subgroup $H'$ is a free subgroup of rank two and is malnormal in $F$;
\item for every $i \in \{1, \ldots, m\}$, every $l \in \{1, \ldots, n\}$ and every element $f$ of $F$, the intersection $a_i^{-1}(H_l) \cap fH'f^{-1}$ is trivial. 
\end{itemize}
By \cite[Theorem 5.16]{Kap99}, the collection $\calm$ is not empty. 
We remark that every element of $\calm$ satisfies conditions (i) and (ii) in Lemma \ref{kap2}. 

For each $i \in \{1, \ldots, m\}$ and each element $H'$ of $\calm$, we put as follows:
\begin{eqnarray*}
\overline{\bbk}_i &=& \{K \subset H' ~|~ K \neq \{1\} \text{ and } K=H'\cap fa_i(H')f^{-1} \text{ for some } f \in F\}; \\
I_1(H') &=& \{i \in \{1, \ldots, m\} ~|~ \overline{\bbk}_i=\emptyset\}; \\
I_2(H') &=& \{i \in \{1, \ldots, m\} ~|~ \overline{\bbk}_i=\{H'\}\};\\
I_3(H') &=& \{i \in \{1, \ldots, m\} ~|~ \overline{\bbk}_i\neq\emptyset\text{ and }\overline{\bbk}_i\neq\{H'\}\}.
\end{eqnarray*}
Since every finitely generated subgroup of $F$ is quasiconvex in $F$ (see \cite[Section 2]{Sho91}), the subgroup $a_i(H')$ as well as $H'$ is quasiconvex and malnormal in $F$. 
By \cite[Theorem 7.11]{Bow12}, the group $F$ is hyperbolic relative to $\{a_i(H')\}$. 
Since $H'$ is quasiconvex in $F$, it follows from \cite[Theorem 1.1 (1)]{MP12} that $H'$ is quasiconvex relative to $\{a_i(H')\}$ in $F$.  
By \cite[Theorem 9.1]{Hru10}, the collection $\overline{\bbk}_i$ has a finite set of representatives of $H'$-conjugacy classes $\bbk_i=\{K_{i,j} ~|~ j \in \{1, \ldots, n_i\}\}$ and $H'$ is hyperbolic relative to $\bbk_i$. 
For every $j \in \{1, \ldots, n_i\}$ the subgroup $K_{i,j}$ is finitely generated and malnormal in $H'$ (see \cite[Propositions 2.29 and 2.36]{Osi06a}). 

For the proof of the lemma, it suffices to show that there exists an element $H$ of $\calm$ such that $I_3(H)$ is empty. 
Indeed if $H$ is such an element of $\calm$, then for every $a\in A$, either both $H\cap a(H)$ and $H\cap a^{-1}(H)$ are equal to $H$ or the intersection $H\cap a(H)$ is trivial. 
If the former occurs, then $a(H)$ is equal to $H$.
Hence the subgroup $H$ has the desired properties.

Let $H'$ be an element of $\calm$. 
We have only to consider the case where the set $I_3(H')$ is not empty. 
Then for every $i\in I_3(H')$ and every $j \in \{1, \ldots, n_i\}$, the group $K_{i,j}$ is a proper malnormal subgroup of $H'$ and hence it is of infinite index in $H'$ by Lemma \ref{aminfind}. 
By \cite[Theorem 5.16]{Kap99}, there exists a proper subgroup $H''$ of $H'$ satisfying the following: 
\begin{itemize}
\setlength{\itemsep}{0mm}
\item $H''$ is a free subgroup of rank two and is malnormal in $H'$;
\item for every $i \in I_3(H')$, every $j \in \{1, \ldots, n_i\}$ and every element $h'$ of $H'$, the intersection $K_{i,j} \cap h'H''h'^{-1}$ is trivial. 
\end{itemize}
Since $H'$ belongs to $\calm$, the subgroup $H''$ also belongs to $\calm$. 

We claim that if $i\in \{1, \ldots, m\}$ belongs to $I_3(H')$, 
then for every element $f$ of $F$, the intersection $H'' \cap a_i(fH''f^{-1})$ is trivial. 
Indeed, the intersection $H' \cap a_i(fH'f^{-1})$ is either trivial or conjugate to $K_{i,j}$ in $H'$ for some $j \in \{1, \ldots, n_i\}$. 
If the former occurs, the claim obviously holds. 
If the latter occurs, the intersection $H'' \cap a_i(fH''f^{-1})$ is conjugate to a subgroup of $K_{i,j} \cap h'H''h'^{-1}$ for some $j \in \{1, \ldots, n_i\}$ and some element $h'$ of $H'$. 
By the choice of $H''$, the intersection $H'' \cap a_i(fH''f^{-1})$ is trivial.

The above claim implies that $I_3(H')$ is contained in $I_1(H'')$. 
Since $H''$ is a subgroup of $H'$, the set $I_1(H')$ is also contained in $I_1(H'')$. 
Hence the union $I_1(H') \cup I_3(H')$ is a proper subset of the union $I_1(H'') \cup I_3(H'')$ if $I_3(H'')$ is not empty. 
By repeating this procedure if necessary, we can find an element $H$ of $\calm$ such that $I_3(H)$ is empty. 
\end{proof}

Now we are ready to prove Theorem \ref{kap1}. 
Given a subgroup $H$ of a group $G$, we put 
\begin{eqnarray*}
V_G(H) &=& \{ g \in G ~|~ H \cap gHg^{-1}\text{ is of finite index both in }H\text{ and }gHg^{-1}\}. 
\end{eqnarray*}

\begin{proof}[Proof of Theorem \ref{kap1}]
It follows from the assumption that $M$ has a finite index normal subgroup $F$ which is a non-abelian free group of finite rank. 
The action of $M$ on $F$ by conjugations induces a homomorphism from $M$ to $\Out(F)$. 
We denote the image of this homomorphism by $U$. 
Since $F$ is a finite index subgroup of $M$, $U$ is a finite subgroup of $\Out(F)$. 
For each $l \in \{1, \ldots, n\}$, we put $H_l=M_l \cap F$. 
Since $M_l$ is finitely generated and $F$ is a finite index subgroup of $M$, this implies that $H_l$ is also finitely generated. 
Since the subgroup $M_l$ is of infinite index in $M$ and the subgroup $F$ is of finite index in $M$, the subgroup $H_l$ is of infinite index in $F$ and of finite index in $M_l$. 

Therefore we can take a subgroup $H$ of $F$ given by Lemma \ref{kap2}. 
We put $V=V_M(H)$. 
By \cite[Theorem 1.6]{H-W09}, the group $H$ is a finite index subgroup of $V$. 
Hence $V$ is a finitely generated and virtually non-abelian free group. 
By the definition of $U$ and condition (iii) in Lemma \ref{kap2}, for every element $m$ of $M$, either $mHm^{-1}$ is equal to $H$ or the intersection $H \cap mHm^{-1}$ is trivial.  
Hence for every element $m$ of $M \setminus V$, the intersection $H \cap mHm^{-1}$ is trivial. 
Since $H$ is a finite index subgroup of $V$, this implies that $V$ is almost malnormal in $M$. 
By condition (ii) in Lemma \ref{kap2}, for every $l \in \{1, \ldots, n\}$ and every element $m$ of $M$, the intersection $H_l \cap mHm^{-1}$ is trivial and hence the intersection $M_l \cap mVm^{-1}$ is finite. 
\end{proof}

\section{Proof of Theorem \ref{hypemb}}
\label{sect-proof}
We prove Theorem \ref{hypemb}. 
As the case of Corollary \ref{sq}, it suffices to show the following in view of \cite[Corollary 4.5]{Osi06b}. 

\begin{Thm}\label{hypemb'}
Let $G$ be a group which is hyperbolic relative to a family $\bbk$ of proper subgroups and $\Gamma$ a subgroup of $G$ which is neither virtually cyclic nor parabolic with respect to $\bbk$. 
If $\Gamma$ contains an element of infinite order, then there exists a finitely generated and virtually non-abelian free subgroup $V$ of $G$ which is hyperbolically embedded into $G$ relative to $\bbk$ and contains $V\cap \Gamma$ as a finite index subgroup. 
Moreover if $G$ is torsion-free, then $\Gamma$ contains a free subgroup of rank two which is hyperbolically embedded into $G$ relative to $\bbk$. 
\end{Thm}

This generalizes a result due to I. Kapovich \cite[Theorem C]{Kap99} for torsion-free hyperbolic groups. 
For the proof of Theorem \ref{hypemb'}, we show the following lemma. 

\begin{Lem}\label{hypemb'fg}
Let $G$ be a group which is hyperbolic relative to a finite family $\bbk$ of proper subgroups and $\Gamma$ a subgroup of $G$ which is neither virtually cyclic nor parabolic with respect to $\bbk$. 
Suppose that $G$ is finitely generated. 
Then there exists a finitely generated and virtually non-abelian free subgroup $V$ of $G$ which is hyperbolically embedded into $G$ relative to $\bbk$ and contains $V\cap \Gamma$ as a finite index subgroup. 
Moreover if $G$ is torsion-free, then $\Gamma$ contains a free subgroup of rank two which is hyperbolically embedded into $G$ relative to $\bbk$. 
\end{Lem}

\begin{proof}
By Lemma \ref{sq'fg}, the group $\Gamma$ contains a free subgroup $F$ of rank two which is strongly quasiconvex relative to $\bbk$ in $G$. 
We put $M=V_G(F)$. 
Since $F$ is strongly quasiconvex relative to $\bbk$ in $G$, it follows from \cite[Theorem 1.6]{H-W09} that $F$ is a finite index subgroup of $M$. 
Hence $M$ is strongly quasiconvex relative to $\bbk$ in $G$ and we have $V_G(M)=M$. 
By \cite[Corollary 8.7]{H-W09}, there exists only finitely many double cosets $MgM$ in $G$ such that $gMg^{-1}$ is not equal to $M$ and the intersection $M \cap gMg^{-1}$ is infinite. 
We denote the collection of such double cosets by $\{Mg_lM ~|~ l \in\{1, \ldots, n\}\}$. 
For each $l \in \{1, \ldots, n\}$, we put $M_l=M \cap g_lMg_l^{-1}$. 

We claim that for each $l \in \{1, \ldots, n\}$, $M_l$ is of infinite index in both $M$ and $g_lMg_l^{-1}$. 
Indeed, assume that this does not hold. 
Since $g_l$ does not belong to $M$ and $V_G(M)$ is equal to $M$, the subgroup $M_l$ is of infinite index in either $M$ or $g_lMg_l^{-1}$. 
By replacing $g_l$ by its inverse if necessary, we may assume that the subgroup $M_l$ is of finite index in $M$ and of infinite index in $g_lMg_l^{-1}$. 
It follows from \cite[Lemma 6.6]{Kap99} that for every positive integer $k$, the subgroup $M \cap g_l^kMg_l^{-k}$ is of finite index in $M$ and of infinite index in $g_l^kMg_l^{-k}$.
Then for every positive integer $p$, the subgroup $\bigcap_{k=1}^p M \cap g_l^kMg_l^{-k}$ is of finite index in $M$ and hence the intersection $\bigcap_{k=1}^p g_l^kMg_l^{-k}$ is infinite. 
By \cite[Theorem 1.4]{H-W09}, there exists a positive integer $p$ such that $g_l^p$ belongs to $M$. 
This contradicts the assumption that the subgroup $M\cap g_l^pMg_l^{-p}$ is of infinite index in $g_l^pMg_l^{-p}$. 

We also claim that for every $l \in \{1, \ldots, n\}$, the subgroup $M_l$ is finitely generated. 
Indeed, since $M$ is strongly quasiconvex relative to $\bbk$ in $G$, it follows from Theorem \ref{sru-sq} that the conjugate $g_lMg_l^{-1}$ is also strongly quasiconvex relative to $\bbk$ in $G$. 
Hence $M_l$ is strongly quasiconvex relative to $\bbk$ in $G$ (see for example \cite[Theorem 9.8]{Hru10} and \cite[Theorem 4.18]{Osi06a}). 
By Theorem \ref{sru-sq}, the subgroup $M_l$ is finitely generated. 

Hence it follows from Theorem \ref{kap1} that there exists a finitely generated and virtually non-abelian free subgroup $V$ of $M$ such that $V$ is almost malnormal in $M$ and $mVm^{-1} \cap M_l$ is finite for every $l \in \{1, \ldots, n\}$ and every element $m$ of $M$. 
Since $M$ contains a finitely generated free subgroup of finite index and $V$ is a finitely generated subgroup of $M$, the subgroup $V$ is undistorted in $M$, that is, it is strongly undistorted relative to the empty family $\emptyset$ in $M$. 
Since $M$ is strongly undistorted relative to $\bbk$ in $G$ by Theorem \ref{sru-sq}, the subgroup $V$ is strongly undistorted relative to $\bbk$ in $G$. 

We claim that $V$ is almost malnormal in $G$. 
Indeed, assume that $V$ is not almost malnormal in $G$. 
Then there exists an element $g$ of $G \setminus V$ such that the intersection $V \cap gVg^{-1}$ is infinite. 
In particular the intersection $M \cap gMg^{-1}$ is also infinite. 
Since $V$ is almost malnormal in $M$, the element $g$ belongs to $G \setminus M$. 
Then for some $l \in \{1, \ldots, n\}$ and some elements $m_1$ and $m_2$ of $M$, the element $g$ is equal to $m_1g_lm_2$. 
Therefore the intersection $V \cap (m_1g_lm_2)V(m_1g_lm_2)^{-1}$ is infinite and hence the intersection $m_1^{-1}Vm_1 \cap g_lMg_l^{-1}$ is also infinite. 
This contradicts the condition that for every $l \in \{1, \ldots, n\}$ and every element $m$ of $M$, the intersection $mVm^{-1} \cap M_l$ is finite. 

Thus the subgroup $V$ is strongly undistorted relative to $\bbk$ and almost malnormal in $G$. 
By Theorem \ref{hypemb-sruam}, the subgroup $V$ is hyperbolically embedded into $G$ relative to $\bbk$. 

For the case where $G$ is torsion-free, we can take a desired subgroup $V$ of $\Gamma$ by applying \cite[Theorem 5.16]{Kap99} to $F$ instead of applying Theorem \ref{kap1} to $M$ in the above argument.
\end{proof}

\begin{proof}[Proof of Theorem \ref{hypemb'}]
The proof is done in the same way as the proof of Proposition \ref{sq'} by using Lemma \ref{hypemb'fg} and Lemma \ref{fgru'} (3) instead of Lemma \ref{sq'fg} and Lemma \ref{fgru'} (2), respectively. 
\end{proof}

\section*{Acknowledgements}
The authors would like to thank Professor Ilya Kapovich for his useful suggestion.

\end{document}